\documentclass[12pt]{article}
\usepackage{geometry,amssymb,amsthm,amsmath,
graphics,enumerate}

\let\oldthebibliography\thebibliography
\renewcommand{\thebibliography}[1]{%
  \oldthebibliography{#1}%
  \setlength{\itemsep}{0pt}%
  \setlength{\parskip}{0pt}%
}

\usepackage{xcolor}
\usepackage{times}
\usepackage{amsfonts}
\usepackage{amscd}
\usepackage{url}
\usepackage{showlabels}

  \usepackage[normalem]{ulem}
\usepackage{soul}
\usepackage{hyperref}

\newcommand{\sidebar}[1]{\vskip10pt\noindent
 \hskip.70truein\vrule width2.0pt\hskip.5em
 \vbox{\hsize= 4truein\noindent\footnotesize\relax #1 }\vskip10pt\noindent}

\newbox\smilebox
\newbox\anchorbox
\newbox\noanchorbox
\newbox\tempbox

\setbox\smilebox=\hbox{$\smile$}

\def\anchor{\hbox{\vtop{
           \hbox to \wd\smilebox{\hfil\vrule width.4pt height7pt depth1pt\hfil}
           \vskip  -11.5truept
           \hbox to \wd\smilebox{\hfil$\smile$\hfil}}}}
\setbox\anchorbox=\anchor
\def\noanchor{\hbox{\vtop{
           \hbox to \wd\anchorbox{\hfil\anchor\hfil}
           \vskip -14truept
           \hbox to \wd\anchorbox{\hfil/\hfil}}}}
\setbox\noanchorbox=\noanchor

\def\fg#1#2#3{\setbox\tempbox=\hbox{$\scriptstyle{#2}$}
\ifnum\wd\anchorbox>\wd\tempbox\dimen255=\wd\anchorbox
\else\dimen255=\wd\tempbox\fi
{#1\,\vtop{\hbox to \dimen255{\hfil\anchor\hfil}
           \vskip -6truept
           \hbox to \dimen255{\hfil$\scriptstyle{#2}$\hfil}}
           \,#3}}

\def\nfg#1#2#3{\setbox\tempbox=\hbox{$\scriptstyle{#2}$}
\ifnum\wd\noanchorbox>\wd\tempbox\dimen255=\wd\noanchorbox
\else\dimen255=\wd\tempbox\fi
{#1\,\vtop{\hbox to \dimen255{\hfil\noanchor\hfil}
           \vskip -6truept
           \hbox to \dimen255{\hfil$\scriptstyle{#2}$\hfil}}
           \,#3}}

\setbox1=\hbox{$\bot$}

\def\north#1#2{#1\,
\hbox{$\bot$\llap {\hbox to\wd1 {\hfil $/$\hfil}}}
\,#2}

\def\nao#1#2#3{#1\  \hbox{\vtop{
\baselineskip=4pt
\hbox{$\bot$\llap {\hbox to\wd1 {\hfil $/$\hfil}}
\hskip .05em \llap{\hbox{$^{\scriptscriptstyle{a}}$}}}\hbox{$\scriptstyle
{#2}$}}}\, #3}

\def\abar{\overline{a}}
\def\bbar{\overline{b}}

\def\dbar{\overline{d}}

\def\hbar{\overline{h}}

\def\mbar{\overline{m}}

\def\vbar{\overline{v}}

\def\ybar{\overline{y}}

\def\Qbar{\bf Q}
\def\Rbar{\bf R}

\def\tp{{\rm tp}}

\def\bK{{\bf K}}

\def\C{{\mathfrak  C}}

\def\FF{{\bf F}}
\def\C{{\cal  C}}

\def\FF{{\bf F}}

\def\K{{\bf K}}

\def\Z{{\mathbb Z}}

\def\RR{{\mathbb R}}
\def\tp{{\rm tp}}

\def\acl{{\rm acl}}
\def\dcl{{\rm dcl}}

\def\Fa0{{\FF^a_{\aleph_0}}}

\def\<{\langle}
\def\th{{\mathop {Th}}}

\def\>{\rangle}

\def\Lscr{\mathcal L}

\def\phi{\varphi}
\def\PA{PA}

\def\true{true}
\def\Const{Const}
\def\C_{C_}

\def\DN_{DN}

\makeatother
\newtheorem{theorem}{Theorem}[section]
\newtheorem{assumption}[theorem]{Assumption}
\newtheorem{proposition}[theorem]{Proposition}
\newtheorem{definition}[theorem]{Definition}
\newtheorem{notation}[theorem]{Notation}
\newtheorem{remark}[theorem]{Remark}
\newtheorem{example}[theorem]{Example}
\newtheorem{lemma}[theorem]{Lemma}
\newtheorem{corollary}[theorem]{Corollary}
\newtheorem{fact}[theorem]{Fact}

\newcommand{\qtp}{\mathrm{qtp}}

\begin{document}

\author{John T. Baldwin, Constantin C. Br\^incu\c s}

\title{Categoricity for an inferential $\omega$-logic and in $L_{\omega_1,\omega}$}
\maketitle 

\begin{abstract} This paper provides two extensions of first order logic by `$\omega$-rules'. In each case we characterize the countable structures whose theory in the logic is categorical (has a unique model). In the one-sorted inferential $\omega$-logic, both Robinson's system $Q$ and Peano Arithmetic become categorical. In the two-sorted generalized 
$\omega$-logic we show
 each complete $L_{\omega_1,\omega}$ sentence defines the same class of structures
as a first-order theory with the appropriate $G-\omega$-rule. The results depend on proving that the inferential rules for the logics are categorical, i.e. they uniquely determine certain truth-conditions for the logical connectives and quantifiers.

KEYWORDS: categoricity, inferentialism, first-order logic, first-order theories, $\omega$-rules, $L_{\omega_1,\omega}$.\end{abstract}

\section{ Introduction}\label{intro}
This paper exploits diverse meanings of `categoricity' in disparate areas of logic.
On one hand, there is the conviction that the theory of canonical structures like arithmetic or the real numbers
ought to be axiomatized by a categorical theory  (in some logic), one that has a unique model (up to isomorphism).
On the other, logical inferentialists argue that the truth conditions of the logical operators should be determined by
the rules of inference; that is, logical calculi should be categorical in a different sense (Definition~\ref{langdef}.3b). However, as early as \cite{CarnapFor}, Carnap gave conflicting assignments of truth values for 
propositional  and quantified formulas that each satisfied the same rules of inference.

Since at least Benacerraf and Putnam, the notion of determinacy of reference for canonical mathematical
structures has been a central topic in the philosophy of mathematics.  One road has been the structuralist attempt to avoid the entire (even naive) set theoretic notion of model. Most attention has focused on choosing a 
logic whose theory of that intuitively well-conceived  structure has a unique model.

Naturally, Dedekind's second order axiomatization of arithmetic \cite{Dedekind} is the standard to measure against. Our discussion of the philosophical implications of these results in \cite{BaldwinBrincusII} builds on \cite[Part B]{ButtonWalshbook} which critiques both second-order and 
 weaker logics' abilities for this task.  First-order logic is properly dismissed immediately because of the 
upward L\"{o}wenheim-Skolem Tarski theorem. 
 
The major innovations of this paper  provide  modified
 inferential definitions of the $\omega$-rule that extend  $L_{\omega,\omega}$ and categorically characterize countable structures as in $L_{\omega_1,\omega}$.

{\em The Inferential $\omega$-rule} (the $I-\omega$-rule; Definition~\ref{ourinfrule})  allows (putative) non-standard parameters to appear in  the formula in its hypothesis. Adding this rule yields a categorical interpretation of the universal quantifier. Further, we  show Theorem~\ref{wrulebworks}: Peano arithmetic (even more, Robinson's system $\Qbar$) is categorical with
 our $I-\omega$-rule. 

 Working in two-sorted $G-\omega$-logic ($G$ for generalized) we have Theorem~\ref{omegaruleatomic2}: Each 
$L_{\omega_1,\omega}(\tau)$-complete (Scott; Definition~\ref{mathback}.1)  sentence $\phi$ is `structurally equivalent with' (same $\tau$-models as) a pseudo-elementary class (the relativized $\tau$-reducts of an associated first-order theory ${\hat T}^\phi$ in the $G-\omega$-logic).  Thus, such a sentence is categorical if and only if the associated ${\hat T}^\phi$   is in the $G-\omega$-logic.
 Thus, the philosophical issue of the existence of infinite formulas is reduced to the issue of the reliability of infinitary proofs in ordinary first order logic. For each of our rules,  we characterize those countable structures that admit a categorical description in the infinitary logic.

We structure the paper as follows: in 
Section~\ref{rules valuations} we set up the terminological and conceptual settings 
of the paper and introduce the main ideas of an inferentialist approach to logic. In Section~\ref{subquant} we introduce a  much more expressive  $\omega$-logic, $I-\omega$-logic, and prove that its rules of inference are categorical. We then prove that $PA$ plus the $I-\omega$-rule, $\PA^{\omega}$, is categorical.
In Section~\ref{infinitary},  relying on certain results of Scott, Chang, Vaught,  and Morley we obtain the characterization of categoricity in $L_{\omega_1,\omega}$ described above.

\section{Rules, Valuations and Logical Inferentialism}\label{rules valuations}
\stepcounter{subsection}

From a logical inferentialist point of view,
the formal rules of inference  in a given deductive system and vocabulary 
should uniquely determine the meanings (i.e. the truth-conditions)
of the logical connectives and quantifiers from a language.
That is, an admissible (sound) valuation on a set of sentences makes exactly the same sentences true as those true
in associated structure (Lemma~\ref{modval}) with the standard Robinson semantics used throughout this paper (Notation ~\ref{langdef}.2).




\begin{definition}[Vocabularies and structures]\label{str}
\begin{enumerate}
 \item By a  vocabulary $\tau$ we mean a set of relation, and function symbols with prescribed arity. A $0$-ary
function is called a constant. 
These are not logical symbols,
but vary with the topic being formalized. For any vocabulary $\tau$, a $\tau$-structure $M$\footnote{More pedantically one should denote the structure by a different font, e.g. 
$\mathfrak{M}$. Following current model theoretic notation, we simply overuse $M$. Model theorists typically write $M$ is a `model' of $T$, and `$\tau$- structure' when not dealing with a particular theory.}  
consists of a domain $M$  equipped with: for
each $n$-ary relation symbol $R$  in $\tau$ a subset $R^M \subseteq M^n $, the sequences that satisfy $R$,
and for each $n$-function symbol $f$ a function $f^M$ mapping $M^n$ into $M$.

\item Our theories are all formulated starting from a countable base vocabulary $\tau$. To define truth on an uncountable structure $M$, we allow uncountably many auxiliary constants naming the elements of $M$
(Definition~\ref{varcon}) in an expanded vocabulary $\tau_M$. By the downward L{\"o}wenheim-Skolem theorem  our arguments can take place in countable models so each inference rule invokes only countable sets of constants.
\end{enumerate}
\end{definition}

\begin{definition}[Permissible Valuations] \label{trv} {\rm
\begin{enumerate} 
\item A {\em valuation} is a function from a set $\Sigma$ of sentences to truth values.

\item A {valuation} is {\em permissible} if it takes 
values on all quantifier free sentences in a given vocabulary, the resulting set of 
sentences is consistent\footnote{In particular, this prevents three sentences 
$f(a) =b, f(a) =c, b \neq c$ all being given the value true.}, and it preserves the soundness of the inference rules for identity and for propositional logic.

\end{enumerate}}
\end{definition}

Permissibility significantly narrows \cite{Garson}'s general notion of valuation
as any function from sentences to truth values.
In fact, since each permissible valuation $v$ gives a 
 structure $M$ exactly when for any  
$\abar = \langle a_0, \ldots a_n\rangle$, $\abar \in R^M$ if and only if $ v(R(\abar))= true$
and similarly for functions, we have:

 \begin{lemma}\label{modval} Each
 permissible valuation $v$ gives a $\tau$-structure for the vocabulary $\tau$ whose symbols appear in the domain of $v$. All permissible  valuations  on a $\tau$-structure  $M$ give the same values to the quantifier free sentences of $\tau$.

 \end{lemma}

Note that the extension from `same values on quantifier free sentences' to `on all sentences' is not automatic. It will follow for our inferential logics only after we prove the quantifiers are
inferentially determined.  For any vocabulary $\tau$ 
 in either the one or two-sorted case, the collection for $\tau$-structures permissible for either of our logics (and the collection of sentences) is exactly the same as in first order. However, there are fewer admissible (Notation~\ref{langdef}.3.a) than permissible valuations; this restriction fuels the categoricity argument.


Definition~\ref{langdef}.2  is  the normal model theoretic
definition of truth, (\cite{Markerbook, Shoenfield, ButtonWalshbook}; we refer to  it as $\bf R$  (Robinson) semantics.\footnote{\cite{ButtonWalshbook} distinguish between the Tarski (\cite{TarskiVaught, ChangKeisler}  and Robinson (\cite{Shoenfield, Markerbook, Sacks}) semantics depending on whether the domain of an assignment is variables or constants, but note that they yield the same truth values in each model for sentences \cite[1.7,1.8]{ButtonWalshbook}. \cite{ButtonWalshbook} label as a `hybrid approach' a Robinson approach in which one is careful to avoid identifying the names with the elements of the models that they name. Moreover, they suggest that there is an induction over languages and models, since they introduce the constants as needed. For us, the truth of sentences on $M$ and $N$ is defined independently on each structure as we emphasize by the notation $\tau_M$  in our definition of truth in a structure. We do insist on their requirement that the names of a elements of a model $M$ cannot be elements of $M$. Since this is only a minor modification of Robinson,
 we omit the word hybrid and write Robinsonian semantics (including R-semantics, R-valuation).} However, we represent the satisfaction relation
on a structure $M$ by a valuation $v_M$ as in Definition~\ref{trv}

 \begin{notation}\label{langdef} {\rm Logics: 
 
 A logic $\Lscr$ consists of the lexical items defined below and 
 two maps delineating the true (in a model) sentences (2) and the provable ones (3).
 
 \begin{enumerate}
 \item Syntax:
 \begin{enumerate}
 \item In this paper,  a logic $\Lscr$ specifies logical operators including   propositional connectives $\&$, $\vee$, $\sim$ and the first order quantifiers\footnote{Of course,  \cite{modthelog} deals with many generalized quantifiers; but here all quantifiers are first order since we consider only $L_{\omega,\omega}$ and $L_{\omega_1,\omega}$, which allows countable
conjunctions and disjunctions, while modifying the inference rules for $L_{\omega,\omega}$. While in \cite{modthelog},  the logics considered need not have rules of inference, finding   categorical such rules is one of our goals.} ($\forall$, $\exists$) with a list of symbols for variables.

\item For $\tau$ as in Definition~\ref{str},  $\Lscr(\tau)$-formulas are defined by induction as in \cite[\S 1.2]{ButtonWalshbook}.

\item An $\Lscr(\tau)$ theory is a consistent set of 
$\tau$-sentences. A {\em pseudo-elementary} class\footnote{The terminology dates to Tarski in the early 1950's.  \cite[Chapter VII]{Shelahbook}} of $\tau$ structures
is the collection of $\tau$ reducts of a theory in larger vocabulary $\tau'$.

\end{enumerate}

\item First order semantics:  

Fix a function $f_M$    from the constants in an expanded vocabulary $\tau_M$ that adds a name for each element of $M$, such that $M \models R(\abar)$ if and only if
    $f_M(\abar) \in R^M$. This gives a truth value to each atomic formula ( $v_M(R(\abar)) = T$)  and thus
   to each quantifier free $\tau_M$-formula (using classical truth table rules for the propositional connectives).
We write $M\models R(\abar)$ if and only if $v_M(R(\abar)) = T$. 
   
    This collection of sentences is called the
   {\em quantifier-free diagram} of $M$, $\Delta_0(M)$.  Thus $v_M(R(\abar)) = T$ is a permissible valuation on $\Delta_0(M)$.

This definition\footnote{The definition of truth here corresponds to Robinson-Tarski-Geach semantics (see \cite[Sec. 4]{Lavine}. }
is extended to define truth in $M$ for arbitrary $\tau^M$-formulas by induction on quantifiers as in \cite[p.17]{ButtonWalshbook}.  The resulting $\Delta(M)$ is called the {\em complete diagram}.  In particular, for a  sequence
of constants $\abar$ from $\tau_M$ and a formula of quantifier rank $n$, $M \models (\forall x)\phi(x,\abar)$ if and only $M\models \phi(c,\abar)$ for each $c\in \Const(\tau_M)$ . Thus, by induction, we have for all formulas, $M \models \phi$ and $v_M(\phi) = T$ are equivalent.

\item  Inferential logic: An inferential logic contains a
set ${\bf  R}$ of rules of inference governing the logical operations (propositional connectives and quantifiers). 
\begin{enumerate}

 \item A {\em permissible} valuation $v$ (Definition ~\ref{trv}) over a vocabulary $\tau$ is {\em admissible} if and only if it preserves the soundness of the relation of logical derivability defined by the set ${\bf  R}$,  i.e. for every rule ${\bf r} \in \bf    R$, if $v$  assigns truth to the premise of an instance of ${\bf r}$, it also assigns truth to the conclusion.  In particular, it assigns truth to each logical theorem in $\Sigma$.

  We consider  natural deduction introduction and elimination rules of inference for each propositional connective and each quantifier. These are mostly standard so we specify explicitly only those that are central to our argument\footnote{See \cite{Gentzen}. Elegant systems of natural deduction for the propositional connectives and for the first-order quantifiers are given, for instance, in philosophical logic textbooks as \cite [Ch. 4, Ch. 6]{Forbes} and \cite[Sec. 20-3, 31-3]{Smith}.  In contrast, mathematical
logic texts generally rely on Hilbert style rules of inference and an unconscious `Turing thesis' for proofs --all acceptable proof systems give the same validities.}. We specify below the two sets of rules for first order logic that are relevant to our argument. We take the sub-derivations from assumptions, present in the rules of inference that discharge assumptions, to embody \textit{a formally logical derivability relation}\footnote{For cognoscenti:  and not just to be \textit{de facto truth preserving} (see \cite[pp. 399-400]{PelletierHazen}).}, i.e. we use what \cite{Garson} calls \textit{deductive models}.

All the logical calculi studied in this paper
agree on the semantics for the {\em quantifier free} formulas  as defined in 
Notation ~\ref{langdef}.2.
An admissible valuation $v_M$ for a set of inference rules ${\bf  R}$ will make $M$ a model of a collection of sentences $\Sigma$ formalized in the inferential logic defined by ${\bf   R}$. 

\item A set of logical inference rules is {\em categorical}\footnote{The term `categorical' is used in this paper in three related contexts. As above, as a property of an inference system, the usual philosophical use asserting that a theory (in some logic) has exactly one model, and the model theoretic notion of `$\kappa$-categoricity' that a theory has exactly one model of cardinality $\kappa$.} if and only if, for all countable vocabularies $\tau$ and all sets of sentences $\Sigma$, all admissible valuations on $\Sigma$ for those rules provide the propositional connectives, universal, and existential quantifiers with the same truth conditions\footnote{These truth conditions are usually spelled out in most of the  {\em philosophical} logic textbooks (see e.g., \cite[Ch. 5]{Halbach}, \cite[Ch.6]{Forbes}, \cite[pp. 213-14]{Garson}) in terms of the classical truth tables for the propositional connectives and in terms of `objectual' (Tarski) or `substitutional' (Carnap/Robinson) interpretations for the first-order quantifiers. In this paper we use the ideas and notations of the R-semantics (Definition~\ref{langdef}.2) for expressing the truth-conditions for quantifiers.}. For a categorical system the soundness of the rules   constrain each admissible valuation so that, any two admissible valuations that agree on the assignment of truth values to atomic sentences from $\Sigma$ agree on each first order sentence in $\Sigma$.\footnote{Any permissible valuation will give the same $\Delta_0(M)$.
But, various admissible valuations
(Example~\ref{non-categoricityFOL}) for different sets of rules
may give different truth values than the, semantically defined,
$\Delta(M)$. Since the semantic method will apply to any admissible valuation in Sections~\ref{subquant} and \ref{infinitary}, the valuation will agree with the semantic definition. Due to the downward L\"{o}wenheim Skolem this continues to hold  when our rules (Section~\ref{subquant}) force all models to be countable and even when (Section~\ref{infinitary}) uncountable models are allowed. Thus we can recast categoricity as: If a set of rules is categorical, then for any $\tau$   each permissible valuation $v$ corresponds to a structure $M =M_v$ with quantifier free diagram $\Delta_0(M)$ and if $v, v'$ are admissible  with
$\Delta_0(M_v) = \Delta_0(M_v')$ then $\Delta(M_v) = \Delta(M_v')$.}

\end{enumerate}

\end{enumerate}
}
\end{notation}

For the benefit of model theorists, we provide a background on the notion of `categoricity of inference rules'. This categoricity is a central step (Propositions~\ref{infRmeaning} and \ref{G-infRmeaning})  in the main arguments.

\begin{definition}[Varieties of Constants] \label{varcon}
We consider several varieties of constants: {\em $\tau$-constants, auxiliary constants, constant terms}.
A $\tau$-constant is a constant symbol in the vocabulary $\tau$; the $R$-semantics introduces {\em auxiliary constants}  in the
vocabularies $\tau_M$. Either case may also give rise to {\em constant terms}; these are definable unary functions of definable singletons, either a $\tau /\tau_M$ - constant  or  the solution of a formula with no parameters
that has a unique solution.  E.g. the formula $\neg \exists y S(y) =x$ in the theory of an $\omega$-sequence
in the vocabulary $\langle S,=\rangle$. Depending on context `constant' may refer to any of these.
\end{definition}

\begin{example}{(Categoricity of $\& I$ and $\&E$-rules)}\label{conj}{\rm

The introduction and elimination rules of inference for $\&$ enforce that the logical connective `$\&$' behaves according to the `classical' truth table definition. In any admissible valuation for these rules, the conjunction ``$\phi \& \psi$” is true if and only if both ``$\phi$” and ``$\psi$” are true. If ``$\phi \& \psi$” is true, then both ``$\phi$” and ``$\psi$” are true (by $\&E$), and if both ``$\phi$” and ``$\psi$” are true then ``$\phi \& \psi$” is true (by $\&I$). In the sense  defined in Notation~\ref{langdef}.3, these rules are categorical.
However, most of the inference rules for the other propositional connectives are non-categorical. In particular, the trivial valuation that assigns truth to all sentences will make $\phi$ and $\sim\phi$ both true and the provability valuation that assigns truth to each propositional theorem and falsehood to each non-theorem will make $\phi$ and $\sim\phi$ both false and their disjunction true (see \cite{BrincusS}).}
\end{example}

It is well known that a categorical formalization of classical propositional logic depends both on the format of the rules and on how their expressive power is defined. In particular, \cite{CarnapFor}, \cite{ShoSmiley}, and \cite{Garson} agree that a multiple-conclusion\footnote{The conclusion of a multiple conclusion rule is a set of formulas; any of them may be selected.} formalization (originated by Carnap; see \cite[p.348]{Brincus2}) uniquely determines the classical truth-conditions of the propositional connectives while the `axiomatic' (i.e. Hilbert style) proof systems fail for this task. 

\begin{assumption}
\label{propass} We fix a categorical multiple-conclusion  system for propositional logic.
\end{assumption}

The problem of a categorical formalization of first-order logic (FOL) is far more complex.\footnote{See \cite{Garson}, \cite{Brincus} and \cite{Speitel} for details on this problem and on the inferentialist program.} The natural deduction introduction ($\forall I$, $\exists I$)  and elimination rules ($\forall E$, $\exists E$) for the first-order quantifiers are the following:

\begin{definition}\label{standrules}
Let $\tau$ be a first-order vocabulary, $\phi(t)$ be any sentence that contains the individual constant $t$, and $x$ a variable in $\Lscr$.

$$\forall I:  \frac{\overset{\vdots}\phi (t)}{(\forall x)\phi(x)};
\hspace{.1in} \forall E:       \frac{(\forall x) \phi(x)}
        {\phi (t)}.$$
Restrictions on $\forall$I: $t$ does not occur in any premise or assumption on which $\phi(t)$ depends, and $x$ does not occur in $\phi(t)$. $x$ replaces all and only occurrences of $t$ in $\phi(t)$.

$$\exists I: \frac{  \phi (t)
}{(\exists x)\phi (x)};
\hspace{.1in} \exists E: \frac{(\exists x) \phi(x) \hspace{.1in} \overset{[\phi(t)]}{\overset{\vdots}{\psi}}}{\psi}.$$

Restrictions on $\exists$E: $t$ does not occur in i) $(\exists x) \phi(x)$; ii) $\psi$,  and iii) any premise or assumption on which the upper $\psi$ depends (excepting the assumption $\phi(t)$). The square brackets indicate that $\phi(t)$ is an assumption that the $\exists$E-rule discharges. $\phi(t)$ is obtained from $\phi(x)$ by replacing all and only occurrences of $x$ with $t$. 

\end{definition}

Both Carnap (\cite[pp.231-32]{CarnapLog},\cite[p.140, 148-50]{CarnapFor}),  and Garson (\cite[p.216, Th 14.3]{Garson}) showed that, due to the finitary character of the rules of inference, there are admissible valuations for the standard formalizations of first-order logic in which the first-order quantifiers have non-standard truth-conditions. 

\begin{example}[Non-categoricity of $\forall I$ and $\forall E$-rules]\label{non-categoricityFOL}

{\rm
(Carnap) Consider a vocabulary that has only two unary predicates $F$ and $G$ and a countable number of individual constants. 
\begin{enumerate}
\item Let $v$ be a standard ($\bf R$-semantics) truth-theoretic valuation. In particular, it interprets $(\forall x)F(x)$ such that $v((\forall x)F(x))= T$ and  $v(F(c)) = T$ for each constant $c$; the values of  $v(G(c))$ are irrelevant, for any constant $c$.
\item Define a valuation $v'$ that maps the individual constants in the vocabulary onto the objects from a denumerable domain and agrees with $v$ on the quantifier-free formulas. Interpret $(\forall x)F(x)$ in $v'$
as  \textit{`$v(F(c)) = T$ for any constant $c$ and $v'(G(c)) = T$, for a particular $c$, namely $b$'}. In $v'$ the universal quantifier has a richer content than that given by the usual truth conditions in $v$. We consider two options for how $v$ (and thus $v'$) interprets $Gb$, and thus how $v'$ interprets $(\forall x)F(x)$, relative to the corresponding structure. 
\begin{enumerate}
\item If $v(G(b))= v'(G(b))$ is true, then the $\forall E$ and $\forall I$ rules preserve their soundness under $v'$ since $Gb$ does not provide a counterexample. 
\item If $v(G(b))$ is false, then $\forall E$ is sound in $v'$ since its premise, i.e. $v'((\forall x)F(x))$, is false. Likewise, if $v(G(b))$ is false, then the $\forall I$-rule preserves its soundness since both $F$ and $G$ are atomic predicates and there is no generalizable or free variable proof for $Fc$. 
\end{enumerate}
Since $v'$ preserves the soundness of the $\forall I$ and $\forall E$-rules, $v'$ is admissible, although it interprets the universal quantifier non-standardly.

\end{enumerate}
}
\end{example}

\begin{remark} {\rm
 Carnap's non-standard valuation $v'$ is formulated in substitutional terms for a fixed countable vocabulary $\tau$ that involves no extension $\tau_M$. This valuation  offers an alternative to Notation~\ref{langdef}.2 for assigning truth-conditions to a universally quantified sentence. The rules of inference for the universal quantifier preserve their soundness in all structures of the kind mentioned above when $(\forall x)F(x)$ in $v'$ is interpreted as  \textit{`$v(F(c)) = \true$ for any constant $c$ and $v'(G(c)) = \true$, for a particular $c$, namely $b$'}. The valuation $v'$ is given by different semantic rules for the universal quantifier in the same structure.} 
\end{remark}

\begin{remark}[$\omega$-logic, $\omega$-rule, $\omega$-model and standard model] \label{omega-logics and rules}
{\rm
These terms appear with various definitions in various areas of logic. We attempt to differentiate. In all cases we let $c$ range over a countable set $C$ of constants. 
\begin{enumerate}
\item $\omega$-logic and $\omega$-rule. 
There are both  model-theoretic   and   proof-theoretic approaches to $\omega$-logic.\footnote{While this notion developed from the study of arithmetic, this generalization has no reliance on arithmetic.}
\begin{enumerate}
\item  
The more traditional proof-theoretic one  adds the `classical' $\omega$-rule to the rules of inference of {\em one-sorted} first-order logic:

$$\omega\text{-rule}:  \frac{\{\phi (c): c \in C\}}{(\forall x)\phi (x)}$$\label{classwrule}

 In the classical $\omega$-rule, $\phi(x)$ stands for any well-formed formula of the  vocabulary $\tau$ in which the rule is formulated. Likewise, the variable `x' is the only free variable in $\phi$. 
 The $c$ are individual constants\footnote{Instead of using individual constants $c$, the $\omega$-rule is sometimes formulated by using the so-called standard numerals. Since our interest in this paper goes beyond PA, we shall use the formulation of the $\omega$-rule with individual constants $c$} from  $\tau$.
 
 This version of the rule was also used by Carnap \cite[p.38, DC-2]{CarnapLog}, \cite[p.140, D30-3]{CarnapFor} and by Rosser \cite[p.129]{Rosser}, who calls it `Carnap's Rule' -for historical remarks on different uses of the $\omega$-rule see \cite[pp. 101-5]{Isaacson}. Precise inferential formulation of this rule will be given in Definition ~\ref{ourinfrule}.

 \item The model-theoretic version, stemming from (\cite{Henkinomt, Orey, MorleyLeeds} as treated in  \cite[pp.28, 39]{Ebbing}, and \cite[pp. 153-155] {Shapiro}) is {\em two-sorted} with a designated predicate containing
 a fixed countable set of constants. In this case, an $\omega$-model $M$ is one where the predicate $N(M)$ consists only of a designated set of constants.  Thus,  the notion of $\omega$-model is the same as in set theory
 \cite[p 145]{Kunen}, where $N(x)$ is replaced by $x\in \omega$ (So, not 2-sorted but with a distinguished predicate.).

  And the $\omega$-rule is:
  $$\text{Generalized}\hspace{.1in} \omega\text{-rule}:  \frac{\{\phi (c): c \in C\}}{(\forall x )(N(x) \rightarrow\phi(x))}.$$ 
  Chang and Keisler (\cite[p.82-3]{ChangKeisler}) expound both variants. We modify this rule to the $G-\omega$-rule in  Definition~\ref{ourinfrule2}.
  We use the single-sorted approach in Section~\ref{subquant} and the two-sorted in Section~\ref{infinitary} but with a heavy emphasis on the rules of inference in both cases.
  \end{enumerate}

\item $\omega$-model: There are several meanings of this term; the most important distinction  here is whether the context is $1$-sorted or $2$-sorted.
\begin{enumerate}[(a)]
\item $1$-sorted:
\begin{enumerate}[(i)]
 \item {\em $\omega$-sequence  $\omega$-model}
 In the vocabulary $(0,S)$ the domain of the $\omega$-model
 is the set of iterations of applications of a $1-1$ function to $0$.

    \item {\em arithmetic $\omega$-model}
The concept arose in studying arithmetic and the classical notion is: $M$ is an $\omega$-model of $T$ if $T$ interprets arithmetic in the vocabulary $(+,\times,0,1)$ of arithmetic and 
the universe of $M$ is named by the numerals.

\item {\em first order $\omega$-model}
This is generalized by replacing `natural numbers' with `the denotations of a specified countable set of constants' and $T$ is any first order theory. This is the meaning of $\omega$-model in Section~\ref{subquant}.


\end{enumerate}
\item  $2$-sorted: \cite{Orey} introduced the $2$-sorted approach with a predicate $N$, a vocabulary containing countably many constants from $N$
and an $\omega$-model is one where these constants exhaust $N$. This is the meaning of $\omega$-model in Section~\ref{infinitary}.

This notation is especially important in reverse mathematics.
The vocabulary contains two sorts, $N,S$ ($N$ for numbers and $S$ for sets of numbers) and symbols for addition, multiplication, order, $0,1$. An $\omega$ model is one
    in which the structure on the number sort is the standard structure for arithmetic, and the collection of sets is
    non-empty. \cite[p 3]{Simpson} \cite[Section 2.2]{Eastaughsep}.  That is,   second order arithmetic
    with Henkin models rather than full second order quantification is studied.

\end{enumerate}

     \item Standard model:  A {\em standard model} arises when there is an informal notion that has  widely accepted formal counterpart.  E.g.,  arithmetic,  consider a vocabulary $\tau$ that includes the following non-logical terms $\{0, S, +, \times\,, <\}$ with the following abbreviations: $1=S0, 2=SS0, \ldots$. A standard model of arithmetic is an $\omega$-model in which the domain is $N=\{0,1,2, 3,\ldots\}$, i.e. the model omits the set of formulas $\{x\neq 0, x\neq 1, x\neq 2, \ldots\}$ and the functions behave usual.
     
     As \cite[p. 28, 5B]{ButtonWalshbook} point out, a confusion may arise when, influenced by Dedekind's axioms and the omnipotence of second order definability, one takes only $(\omega, 0,S)$ as the standard model of arithmetic.

\end{enumerate}}

\end{remark}

It has been known at least since \cite{Rosser} that the models of $\PA^{\omega}$, i.e. the consequences of the first order Peano axioms  in  first order
logic with the $\omega$-rule, are exactly the models of $TA$ (true arithmetic: all sentences that are true in the standard model of $\PA$). Thus, $\PA^{\omega}$ is not categorical (\cite{Skolem34}).\footnote{See \cite{Shoenfieldomega}, \cite{Franzen}, and \cite{Ketland} for later treatments.}

\section{The Inferential $\omega$-rule and the Categoricity of $\PA^{\omega}$} \label{subquant}
\stepcounter{subsection} 

Notation~\ref{langdef}.2 fixes the  $R$-semantics  as a definition for valuations, interpretations, and structures.
We now describe two systems of inference rules: inferential (I) (here)
and $G$-inferential $(G)$ (\S 4). The admissible valuations for these rules assign truth-conditions for the quantifiers governed by these rules and we prove that the rules uniquely determine these truth-conditions (Propositions ~\ref{infRmeaning}, ~\ref{G-infRmeaning}).

In view of  Assumption~\ref{propass}, to study categoricity of inferential rules, we need consider only the quantifier rules but with auxiliary constants. So  with Assumption~\ref{propass} we pass to the  first order case and even allow
expansions by constants of the original vocabulary.  

 Carnap \cite[p.145]{CarnapFor} introduced a version of the $\omega$-rule (called here the substitutional $\omega$-rule) to provide a categorical formalization of the truth-conditions for the universal quantifier defined over a denumerable domain. He fixes a countable set $\Const(\tau)$ of constants; 
  each valuation $v$ maps these constants {\em onto} a domain (and implicitly the $\tau$ relations and functions). He assumed the  countability of $\Const(\tau)$ and the `onto' condition to enforce the idea that all models are countable. With the Robinsonian semantics (Definition~\ref{langdef}.2)
 we consider valuations over extended vocabularies and thus uncountable models.

 {\em The Inferential $\omega$-rule  we now introduce is considerably stronger than the usual
versions. While the instantiations  that must be satisfied in the hypothesis of the $\omega$-rule are the same as usual, the formula $\phi$ now may have parameters $\dbar$ from $C- \Const(\tau)$. We require 
$\Const(\tau)$ to be countably infinite.}

Definition ~\ref{ourinfrule} below is a syntactic definition of the rules of inference for the quantifiers in the inferential $\omega$-logic. Proposition ~\ref{infRmeaning} shows that the rules of inference for the universal quantifier uniquely determine the truth-conditions for $\forall$,  making the rules in Definition ~\ref{ourinfrule} categorical. While we restrict to vocabularies with countably many constants,  unlike Carnap, we do not require all constants from $C$ to be used in every valuation, but only those from $Const(\tau)$.  Thus, our rules are actually rule schema, indexed by specifying a set of constants.

 We restrict ourselves to vocabularies with countably many constants.  Unlike Carnap, we do not require all constants from $C$ to be used in every valuation. We only require the constants from $Const(\tau)$ to be used in every valuation. Thus, our rules are actually rule schema, indexed by specifying a set of constants.

We state the rule for vocabularies with infinitely many constants.  We can study theories $T$ in vocabularies  with
only finitely many constants such that all models of $T$ have
the definable closure of the $\emptyset$ is infinite. E.g.
$\th(\omega,S,0)$, an $\omega$-sequence.

\begin{definition}{(Inferential $\omega$-rule)}\label{ourinfrule} Fix a vocabulary $\tau$ and use the notation of $R$-semantics as in Notation ~\ref{langdef}.2. For  any countable set of constants $C\supseteq \Const(\tau)$ and any  well-formed formula $\phi(x,\dbar)$ in the vocabulary  $\tau$ expanded by the constants from $C- \Const(\tau)$:

\[I-\omega\text{-rule}: \frac{\bigwedge\{\phi(c ,\dbar): c \in \Const(\tau)\}} {(\forall x)\phi(x,\dbar)};\]

$$I-\forall E:       \frac{(\forall x) \phi(x,\dbar)}
        {\phi (c, \dbar), \text{ for each } c\in C}.$$
 
 The rules for the existential quantifier in the notation of the  R-valuations are the following:

$$I-\exists I: \frac{   \phi (c , \dbar), \text{ for some c $\in$ $C$}}{(\exists x)\phi(x,\dbar)};$$

$$I-\exists E:       \frac{(\exists x) \phi(x,\dbar)}
        {\bigvee\{\phi(c ,\dbar): c \in  C \}}.$$

\end{definition}

\begin{proposition}{(Categoricity of I-$\omega$-rule and $I-\forall E$-rule)}\label{infRmeaning}
The  I-$\omega$-rule and the I-$\forall E$-rule for the universal quantifier in the inferential $\omega$-logic uniquely determine: 

 For any countable set of constants $C\supseteq \Const(\tau)$, any admissible valuation $v_C$, $v_C((\forall x)\phi(x, \dbar))=\true$ iff 
$v_C(\phi(c ))=\true$ for each constant $c\in \Const(\tau)$.

\end{proposition}

\begin{proof}
$(\Longrightarrow)$ For any admissible valuation $v_C$, if $v_C((\forall x)\phi(x, \dbar))=\true$, then $v_C(\phi(c ,\dbar))$ is true, for each $c $ from $C$ (by the I-$\forall$E-rule) and, in particular, for each $\tau$-constant $c$. 

$(\Longleftarrow)$ For any admissible valuation $v_C$, if $\phi(x,\dbar)$ is a formula in the vocabulary $\tau$ expanded by the constants $C$ and $v_C(\phi(c ))=\true$ for each $\tau$-constant $c$, then $v_C((\forall x)\phi(x, \dbar))=\true$ (by the I-$\omega$-rule).

\end{proof}



  We now show that with their reading in terms of  R-valuations the rules of inference for the 
 universal quantifier in the inferential $\omega$-logic make $\PA^{\omega}$ categorical. 
 
A fundamental result for the $\omega$-logic is the $\omega$-completeness theorem (see \cite[p.82]{ChangKeisler}), where $\omega$-model is in the sense of Remark~\ref{omega-logics and rules}.2.a.ii):

\begin{proposition}\label{omegamodexistCK}
A theory $T$ in the vocabulary $\tau$ is consistent in $\omega$-logic if and only if $T$ has an $\omega$-model. 
\end{proposition}

The $\omega$-completeness theorem  tells us that the presence of the $\omega$–rule guarantees the existence of an $\omega$-model, but it does not tell us that {\em all} models of $\PA^{\omega}$  are $\omega$-models (see \cite[Th.2]{Orey}, \cite[Th.3]{Henkinomt} ). Our argument for the categoricity of $\PA^{\omega}$ is formulated from an inferential point of view;  we take the soundness of the $I-\omega$–rule, spelled out in terms of Robinsonian valuations, as primitive. This allows us to read off the truth-conditions of the universal quantifier in the $I-\omega$-logic. With these truth-conditions we provide a positive answer to the question of the categoricity of $\PA^{\omega}$. Roughly, the argument goes as follows:\\

{\bf P1.} The truth-conditions of the universal quantifier are uniquely determined by the $I-\omega$-rule and $I-\forall E$-rule in the inferential $\omega$-logic (Proposition ~\ref{infRmeaning}).\\ 
\indent {\bf P2.} With these truth-conditions of the universal quantifier we obtain a categorical characterization of the natural numbers (Theorem ~\ref{wrulebworks}).\\
\indent {\bf C.} Thus, with the $I-\omega$-rule and $I-\forall E$-rule in the $I-\omega$-logic we obtain a categorical characterization of the natural numbers.\\

The first premise of this argument was established (Proposition ~\ref{infRmeaning}). We show now  the second premise. For clarity, we remind the reader that we work with countable  vocabularies $\tau$ and we assume that all objects from the domain that someone can `refer to' are named by a countable number of individual constants. We prove below  that $\PA^{\omega}$ is categorical. The proof we give below (in Theorem ~\ref{wrulebworks}) shows that a non-standard model for $\PA$ does not satisfy the rules of inference for the universal quantifier in the $I-\omega$-logic.

\begin{definition} \label{prime} 
 A  $\tau$-structure $M$ is an {\em algebraically  or Robinson prime} model of $T$ if it can 
 be embedded in every model of $T$,
\end{definition}

\begin{definition} [Robinson's {\bf Q}]  By $\Qbar$, we mean the finitely axiomatized first-order theory, considerably weaker than Peano arithmetic (PA), whose axioms contain only one existential quantifier. See \cite{TMR} or the splendid wikipedia article on Robinson Arithmetic. 
Like PA, it is incomplete and incompletable in the sense of Gödel's incompleteness theorems, and essentially undecidable. 
The vocabulary $\tau_{\Qbar}$ for $\Qbar$ has a single constant symbol $0$ and function symbols $S$ (unary) and $+, 
\times$ (binary). 
\end{definition}

\begin{notation}  For any first order theory, we denote the set of its consequences under an $\omega$-rule as $T^\omega$. There is a certain ambiguity in the $T^\omega$ notation as its
meaning depends  on which variant of $\omega$-logic is being considered. In this section we use $T^\omega$ as a first order theory closed under the $I-\omega\text{-rule}$ and in the next one as closed under the $G-\omega\text{-rule}$.
\end{notation}

\begin{theorem} \label{wrulebworks} 
{\bf With the inferential $\omega$-rule (Definition~\ref{ourinfrule}), the standard model of Robinson's $\Qbar$ has no proper extension satisfying $\Qbar^{\omega}$, thus
$\Qbar^\omega$ is categorical.}
\end{theorem}

\begin{proof} Note that in any model of $\Qbar$, the $\tau_{\bf Q}$-substructure generated from $0$ is isomorphic to $\bf N =\langle N,0, +,\times\rangle$.
That is, $N$ is an algebraically prime model of $\Qbar$.  Thus,  an arbitrary countable non-standard model $M$ of $\Qbar$ extends (an isomorphic copy of) $N$. 
Fix $\Const(\tau_N)$ as a countable set of constants containing $0$ and a valuation   $v_N$ that maps $\Const(\tau_N)$ onto $N$.
Let $v_M$ be an extension of $v_N$ which enumerates the remaining elements of the $\tau_M$-structure $M$ 
by constants $\Const(\tau_M) \supsetneq 
\Const(\tau_N)$. 
Fix a particular $d \in  \Const(\tau_M)- 
\Const(\tau_N)$ and consider the formula $\phi(x,d): x \neq d$. 

If $v_M$ is admissible, then it satisfies the following instance of the inferential $\omega$-rule.

$$I-\omega\text{-rule}:  \frac{\{\phi (c,d):  c \in \Const(\tau_N)
\}}{(\forall x)\phi (x,d)};
\hspace{.1in}$$

Again, if $v_M$ is admissible, it also satisfies the following instance of the inferential $\forall$E-rule:

$$I-\forall E:       \frac{(\forall x)\phi(x,d)}
        {\phi (e,d), \text{ for each } e \in\Const(\tau_M)}.$$

By this instance of the I-$\forall E$, we have $\phi(d,d)$. But this is a contradiction since $d =d$.  Thus $v_M$ violates the rules for the universal quantifier in $\omega$-logic and is not admissible.

 Note that $v_N$ is clearly admissible since both the $I-\omega$-rule and the $I-\forall E$ rule will take constants $c$ only from $\Const(\tau_N)$ when applied in $v_N$. Thus, the $I-\omega$-rule contains as premises only instances that can be derived in standard first-order logic and no counterexample will occur.
\end{proof}

Of course, this also implies $PA^\omega$ is categorical, since it is a consistent extension of $\Qbar^{\omega}$.

\begin{remark}\label{expmain}{\rm
The proof of Theorem~\ref{wrulebworks} assumes, for reduction to contradiction, that both $v_N$ and $v_M$ are admissible valuations over $\Qbar^\omega$. The $I-\omega$-rule and the $I-\forall$E-rule are legitimately applicable in $v_M$ since the soundness of these rules was defined for arbitrary countable vocabularies $\tau$ and the formula $\phi(x): x \neq d$ is a syntactically well-formed formula in the vocabulary $\tau_M$ of the valuation $v_M$. However, since a contradiction is derived, then $v_M$ is not an admissible valuation. }     
\end{remark}

To give general conditions on a theory $T$ for categoricity with the $I-\omega$-rule requires abstracting two
elements of the proof of Theorem~\ref{wrulebworks}. We made explicit that the theory $\bf Q$
 has an algebraically prime
model.  It was also crucial that every element of that model was named. With this observation it is
straightforward to deduce:

\begin{corollary}\label{apcat} If a first order theory $T$ in a vocabulary $\tau$ has an algebraically prime model 
with every element named by a $\tau$-constant,
then $T$ plus the $I-\omega$-rule
is categorical. 

In particular this applies to the theory of $(\omega,0,S)$
with axioms that $S$ is $1-1$ and every point is in the range except $0$.
\end{corollary}

Without assuming an algebraically prime model the argument for Theorem~\ref{wrulebworks} yields.

\begin{theorem}\label{allmax} Fix a vocabulary $\tau$ with $\aleph_0$ constants and suppose that for any model of $T$, the substructure consisting of the elements named by those constants
is a model of $T$. Then,  no model of $T^{\omega}$ has a proper extension.
\end{theorem}

The following examples show both hypotheses of Corollary~\ref{apcat} are essential. Items 1) and 2) show why we require there be infinitely many constants (or a least
infinitely many elements definable over the empty set) in 
$\tau$. Item 3) show the algebraically prime model is essential.

\begin{example}\label{neccond} {\rm
\begin{enumerate}
    \item 
The pure theory  of an infinite set  (i.e. $\tau = \{=\}$ with axioms asserting there are at least $n$ elements
for every $n< \omega$) shows the necessity  of the second requirement.  It is categorical in all cardinalities,
but the absence of $\tau$-constants make the $\omega$-rule powerless. In contrast,   Theorem~\ref{wrulebworks}  applies
in two related cases:  i) if the vocabulary $\{ =\} \cup C$ with $C$ a countably infinite set of constants with axioms $c \neq d$ for distinct constants in $C$
or ii) the theory of $(\omega, S)$
 in a finite vocabulary even without a constant in $\tau = \{S\}$ (Cf. Definition~\ref{varcon}).

\item Consider the theory of the rational order (dense linear order without endpoints) in the vocabulary $\{<\}$. Like the theory of an infinite set, Theorem~\ref{wrulebworks} does not apply. And like Example 2, there  are  continuum many countable completions each of which is categorical.

\item
The vocabulary has two unary predicates $F,G$ and countably many constants.  The (incomplete) theory $T$ asserts only that the constants are distinct.
Clearly $T$ satisfies Theorem~\ref{allmax}. But there are continuum many countable models
since any distribution of the constants among the four disjoint sets given by the boolean
combinations of $F$ and $G$ yield a  distinct model of $T$.

\end{enumerate}}
\end{example}

{\begin{remark} \label{catrat} Since Example~\ref{neccond} is a canonical example of categoricity (in $\aleph_0$), it presents in  a graphic way the problems of reference. In this case there is a natural choice of naming the constants as $c_{n,m}$ for integers $n$ and $m$.  Namely, we name the rational number $\frac{n}{m}$ with $n$ and $m$ relatively prime by $c_{n,m}$ for $n,m$ integers.
\end{remark}

\section{Categorical  structures in $L_{\omega_1 \omega}$ with an Inferential $G-\omega$-rule}\label{infinitary}
\stepcounter{subsection}

In this section  we  translate $L_{\omega_1,\omega}$-sentences to
theories in first-order logic and add  an inferential $G-\omega$-rule.  This allows us
to translate infinary sentences to finite sentences at the cost of
an infinitary rule of inference.
 Using the $\bf R$-semantics, arbitrarily many constants are available but our central arguments will only involve countable expansions of a
countable base vocabulary.

$L_{\omega_1,\omega}$ extends first order logic by allowing countable conjunctions and disjunctions of countable sets of formulas
$\Phi$ with the restriction that only finitely many free variables can occur in $\Phi$.
Permissible valuations (Definition~\ref{langdef}) now require $M \models \bigwedge \Phi$ if and only 
each $\phi \in \Phi$ is true in $M$
and dually for disjunction.





Those unfamiliar with the model theoretic notion of type are advised to read the classic \cite{Vaughtct} as well as a modern source.

We rely on Fact~\ref{mathback},
whose results are in such sources as \cite{Baldwincatmon,Keislerbook, Markerinflog}, to translate $L_{\omega_1 \omega}$ classes to classes definable in $L_{\omega,\omega}$ with our $G-\omega$-rule.
We  give a detailed proof of a variant of item 4.

Recall that a structure $M$ is an {\em atomic} model of a first order theory $T$ if for every finite sequence $\mbar$ from $M$, there is a
 formula $\phi(\ybar)$ depending only on $p= \tp(\mbar/\emptyset)$ such that $(\forall \ybar)\phi(\ybar) \rightarrow \psi(\ybar)$ for each $\psi(\ybar) \in p$.

\begin{fact}\label{mathback}
\begin{enumerate}
\item Scott's theorem: For any countable structure $A$ for a vocabulary $\tau$, there is an $L_{\omega_1,\omega}(\tau)$ sentence $\phi_A$, {\em the Scott sentence of $A$}, such that all countable models of $\phi_A$ are isomorphic to $A$. $\phi_A$ is {\em complete} in that for any sentence $\psi$ of
    $L_{\omega_1,\omega}$, $\phi_A \vdash  \psi$ or  $\phi_A \vdash \neg \psi$. This holds by the extended completeness theorem (see \cite[Sec. 2]{Bellsep} or \cite[\S 6.1]{Baldwincatmon}), the downward L\"{o}wenheim-Skolem theorem, and the uniqueness of the countable model. 
    
    \item Chang \cite[p 48]{Chang1}: For any  $L_{\omega_1,\omega}$ $\tau$-sentence $\phi$, there  is a vocabulary $\tau^\phi\supset \tau$, a first order $\tau^\phi$-theory $T^\phi$, and a countable collection of types $\Gamma$
    such that  the models of $\phi$ are exactly (in particular, no two non-isomorphic atomic 
    models of $T^\phi$ have isomorphic $\tau$-reducts)  the $\tau$-reducts of the models of $T^\phi$ that 
    omit each type in  $\Gamma$.  Moreover, if $\phi$ is complete, we can consider the reducts of the  atomic models of 
    $T^\phi$.  See \cite[\S 1.2]{Markerinflog} for the first and  Theorem 6.1.12 and Chapter 18 in  \cite{Baldwincatmon} for the atomic case. 

    \item  Applying the classical paper of \cite{Vaughtct} to $T^\phi$, $\phi$ has an uncountable 
    model if and only if   the unique (up to isomorphism) countable  model has a proper submodel 
    isomorphic to itself.

       \item  Morley: \cite{MorleyLeeds} Work in a two-sorted vocabulary $\sigma$ with  a predicate $N$, a countable set $C$ of constants each satisfying $N$, and a predicate $V$ with a $\tau^\phi$-structure from (2) such that $V\restriction \tau^\phi \models T^\phi$. 
       The principal types over the empty set of  finite sequences from $V$ are coded
 in a theory ${\hat T}^\phi$ so that some non-principal type over the emptyset is realized in $V(M)$
        in a model $M$ if and only if the type $\{x\neq c_i \wedge N(x): i<\omega\}$ is
        realized.

        Thus the reducts to $\tau$ of models of $T^\phi$ (from Chang) that arise as the $V(B)$ for  $G-\omega$-models (as defined in Definition~\ref{Morcode}) $B$ of $T^{\hat \phi}$ are exactly reducts of atomic models of $T^{\phi} $ and models of $\phi$. There is no additional rule of inference in \cite{MorleyLeeds}. (See Notation ~\ref{classnot}.5 below for $T^{\hat \phi}$).

\end{enumerate}
\end{fact}

\begin{definition}\label{catdef} We say a {\em complete} sentence $\psi$ of $L_{\omega_1,\omega}$ is {\em categorical} if it has
a unique model (up to isomorphism).
 
\end{definition}
Since $L_{\omega_1,\omega}$ satisfies the downward L\"{o}wenheim-Skolem theorem, necessarily, the unique model is countable. However, {\em mentioning that $\psi$ is a {\em sentence} is essential.}

\begin{example} The $L_{\omega_1,\omega}$ theory of the structure $(\RR,0,1,+, \times, <)$ is categorical; the only model has
cardinality $2^{\aleph_0}$.  The uncountable set of axioms  assert that for each cut in the rationals (individual rationals are named as $\frac{n}{ m}$, although there is no predicate for the set of rationals)
there is a unique point in each cut.
\end{example}



The following notation refers to the classes of models that arise in  the (Fact~\ref{mathback}.2) reduction of complete sentences of  $L_{\omega_1,\omega}$ to atomic models.

\begin{notation}[five classes of models]\label{classnot}
The class of $\tau$-structures that satisfy
\begin{enumerate}
\item $\K_\phi$ is  the class of models of the {\em complete} $L_{\omega_1,\omega}$-sentence $\phi$.
\item $\K_{T^\phi}$ those that satisfy the $\tau^\phi$-theory $T^\phi$.
\item $\K^{at}_{T^\phi}$ is those that are atomic models of the $\tau^\phi$-theory $T^\phi$.
\item $\K^{\tau}_{T^\phi}$ is those that are reducts to $\tau$ of models in $\K^{at}_{T^\phi}$.
\item $\K^{\tau}_{T^{\hat \phi}}$ is the class of $\sigma$-structures in Definition~\ref{Morcode}.
$T^{\hat \phi}$ is the first order theory of $\K^{\tau}_{T^{\hat \phi}}$.

\end{enumerate}
\end{notation}

 The Chang theorem asserts that $\K_\phi=\K^{\tau}_{T^\phi}$. Note that while $\K_{T^\phi}$ has arbitrarily large
models, $\K^{at}_{T^\phi}$ may not; although it does if it has models up to the cardinal\footnote{This cardinal is defined by induction.
$\beth_0 = \aleph_0$, $\beth_{\alpha+1} = 2^{\beth_{\alpha}}$, for limit $\delta$, $\beth_{\delta} = \bigcup_{\alpha<\delta} \beth_{\alpha}$. Morley proved \cite{Morley65a}
there are sentences of $L_{\omega_1,\omega}$ that have models only up to $\kappa$ for any $\kappa <\beth_{\omega_1}$. But any larger and there are arbitrarily large models. } $\beth_{\omega_1}$.

\begin{example}[Distinguishing the classes]\label{ordex} {\rm
Start with the structure of the integers with order $M =(\Z,<,0)$. There are two non-principal types over the empty set for the first order theory of $M$.
$p_{+\infty(x)}$  ($p_{-\infty(x)}$) says $x$ is greater (less) than $0$ and infinitely far away.  
Let $\phi$ be the $L_{\omega_1,\omega}$ sentence characterizing this structure. I.e. the
axioms for discrete linear orders that omit $p_{\pm \infty}$.  $T^\phi$ is the theory of discrete linear order, with  additional symbols $P_+(x)$
($P_-(x)$) with axioms saying $P_+(x)$
($P_-(x)$) implies $p_{+\infty(x)}$  ($p_{-\infty(x)}$) respectively. $T^\phi$ has arbitrarily large models.
But, the unique atomic model of $T^\phi$ omits both types. {Thus,  as asserted in general in Fact~\ref{mathback}.2, $\K_\phi=\K^{\tau}_{T^\phi}$.
}}
 \end{example}

 \begin{example}[Examples of categorical sentences in $L_{\omega_1,\omega}$]\label{catexs}
 \begin{enumerate}
 \item Peano arithmetic: In particular $\phi$ includes $(\forall x)\bigvee_{n<\omega} x = S^n(0)$ and the quantifier-free diagram \\ of $(N,+,\times,0,1,<)$.
 \item The theory of a single bijective function $S$ that has exactly one cycle of length $n$ for each $n$ and no $\omega$-sequence.
\item  The examples of Marcus and Knight \cite[Ex 18.9]{Baldwincatmon} of complete $L_{\omega_1,\omega}$-sentences
 that are categorical (univalent) but the unique (necessarily countable) model $N$ is not homogeneous. In particular, there is no isomorphic
 substructure of $N$.
 \end{enumerate}
 \end{example}

Any sentence of $L_{\omega_1,\omega}$ whose first order consequences satisfy the hypotheses of Theorem~\ref{allmax} has only countable models by
that theorem. In order to consider uncountable models, we turn to the notion of a (generalized) $\omega$-model (Remark~\ref{omega-logics and rules}.1) and to  a new inferential rule modifying  Definition~\ref{ourinfrule}. This variation will  allow examples with no constants in the base language $\tau$ (even $\dcl(\emptyset) =\emptyset$).

 
 \begin{definition} \label{sedef2} Let $\bK$ and $\bK^{\prime}$ be  (pseudo)-elementary classes of models  in the same vocabulary $\tau$, 
 but {\em determined} by theories $T$, $T'$
 in possibly different logics.
We say  the classes $\bK$ and $\bK^{\prime}$ are {\em structurally equivalent}, if they have the same class of models.
\end{definition}

We say `determined' because in the application $\bK'$
is the class of $\tau$-reducts of $T'$.


\begin{definition}\label{Morcode} {\rm Fix a Scott sentence $\phi$ in a vocabulary $\tau$ with a countable model $A$.  
We work in a two-sorted\footnote{Constants and variable will be restricted to specific sorts.} vocabulary $\sigma$ ($\sigma$ depends on $\tau$) with disjoint sorts $(N,V)$ where $N$ contains the set of images of  the constants $\Const(N)=_{\rm df} \langle c_{n,i}:i,n <\omega\rangle$ and $V$ consists of a
$\tau^\phi$-structure satisfying the theory $T^\phi$ constructed from $\phi$ by Fact~\ref{mathback}.2 (Chang).
Each $\tau^\phi$-constant becomes a $\sigma$-constant satisfying $V$.

Adapting \cite{MorleyLeeds}\footnote{Morley coded a single non-principal type;  we code countably many types of arbitrary finite length so we use $n+1$-relations for all $n$, rather a single binary $R$. The constants $c_{n,i}$ satisfying $N(x)$ code all $T^{\phi}$ principal types over $\emptyset$.}, we construct a theory 
 $T^{\hat \phi}$ such that a  $G-\omega$-model
of $T^{\hat \phi}$ omits each non-principal type over $\emptyset$.  In particular, if $B\models T^{\hat \phi}$, the restriction of $B$ to $V(B)$
satisfies $T^{\phi}$. 
Extend the  vocabulary $\sigma$ of $T^{\hat \phi}$ to include,
for each   $n$, an $(n+1)$-ary-relation $R^n$ on $N\times V^n$  in 2-sorted G-$\omega$-logic.
 
The theory $T^{\hat \phi}$
is given by $T^{\phi}$ relativized to $V$ along with axioms saying the $c_{n,i}$ are distinct elements of $N$;
the following axioms on the $R^n$ ensure  that the elements of $N$ code all finite  $\tau^\phi$-types  over the empty set of the
elements of $V$. 
  
 $$ (1)\  (\forall v_0) [R^n(c_{n,i} ,\vbar_0) \leftrightarrow \phi_{i}(\vbar_0)].$$ where
 $\phi_i(\vbar)$ generates the $i$th principal  $n$-type (in $\tau^\phi$) over $\emptyset$ for a given injective enumeration of those types\footnote{ $R^n(c_{n,i},\dbar)$ with $\dbar \in V$ means: $\dbar$ realizes the   $n-\tau^\phi$-type over the empty set indexed by $c_{n,i}$.}. 
    
    $$(2) \ (\forall \vbar_0) [V(\vbar_0) \rightarrow \exists v_1 [N(v_1)\wedge R^n(v_1,\vbar_0) ]$$
We call a model $B$ of $T^{\hat \phi}$ with $\Const(N)= \{c_{n,i}:i,n<\omega\}$ denoting the elements of $N(B)$ a $G-\omega$-model. 
}
\end{definition}

As we now show,  axioms (1) and (2) in Definition~\ref{Morcode} guarantee that in  a model  satisfying the $G-\omega$-rule (Definition~\ref{ourinfrule2}), each non-principal $\tau^\phi$-type over $\empty$ is omitted; thus $V(B)$ is an atomic model of $T^\phi$. Observe that each element of $N$ codes a type over the empty set because of  axiom (2).

Crucially, in Definition~\ref{ourinfrule2}  the instances of the formula in the G-$\omega$-rule must come from the
 index set $N$ and not from the structure being investigated $V$, but the relation $R$ uses parameters from $V$.

Definition ~\ref{ourinfrule2} below is a syntactic definition of the rules of inference for the quantifiers in the  $G-\omega$-logic. Proposition ~\ref{G-infRmeaning} shows that the rules for the universal quantifier uniquely determine the truth-conditions, making the rules in Definition ~\ref{ourinfrule2} categorical (see Notation ~\ref{langdef}.3.b.). As in Section ~\ref{subquant}, we do not require all constants to be used in every valuation. Thus, our rules are actually rule schema, indexed by specifying a set of constants.

\begin{definition}[Inferential $G-\omega$-rule]\label{ourinfrule2} 
$\sigma$ is a $2$-sorted vocabulary as in Definition~\ref{Morcode}; For  any countable set of constants $B\supseteq \Const(\sigma)$ and for any 
$\sigma(B)$-formula $\lambda(x,\vbar)$ and,  in particular, instances $\psi(x,\dbar)$ of a formula $\psi(x,\vbar)$  with $\lg(\vbar) =n$:

{
\small
$$
G\text{-}\omega\text{-rule}:\;
\frac{\bigwedge\{\psi(c_{n,i},\dbar): c_{n,i}\in \Const(N), n,i<\omega\}
\&
(\forall x)(\forall v_i)(\psi(x,\dbar) \rightarrow (N(x) \& \bigwedge_{i< n}V(v_i)))}
{(\forall x)(N(x) \rightarrow \psi(x,\dbar))}
$$
}

$$G-\forall E:       \frac{(\forall x) \lambda(x,\bbar)}
{\lambda (e,\bbar), \text{ for each } e \in B}.$$

        Note that there are constants in $B$ that are not in $N$, but they are not instances of the hypotheses 
        of the $G-\omega$-rule. The G-$\exists I$- and G-$\exists E$-rules are parallel to those in Definition~\ref{ourinfrule}.

\end{definition}

\begin{proposition}{(Categoricity of G-$\omega$-rule and G-$\forall E$-rule)}\label{G-infRmeaning}
The  G-$\omega$-rule and the G-$\forall E$-rule for the universal quantifier in the inferential G-$\omega$-logic uniquely determine:

For any admissible valuation $v_B$, $v_B((\forall x)(N(x)\rightarrow \psi(x, \dbar)))=\true$ iff 
$v_B(\psi(c_{n,i}))=\true$ for each $N$-constant $c_{n,i}$, provided that $v_B((\forall x)(\forall v_i)(\psi(x,\dbar) \rightarrow (N(x) \& \bigwedge_{i< n}V(v_i)))=\true$.

\end{proposition}

\begin{proof}
$(\Longrightarrow)$ 
If  $v_B((\forall x)(N(x) \rightarrow\psi(x, \dbar)))=\true$, then $v_B(N(c_{n,i}) \rightarrow \psi(c_{n,i},\dbar))=\true$, for each $c_{n,i}$ from $\Const(N)$ (by the G-$\forall$E-rule). Thus, $v_B(\psi(c_{n,i},\dbar))=\true$ for each constant $c_{n,i}\in \Const(N)$.

$(\Longleftarrow)$ If $\psi(x,\dbar)$ is a $\sigma$-formula and $v_B(\psi(c_{n,i},\dbar))=\true$ for each $c_{n,i}\in \Const(N)$ and $v_B((\forall x)(\forall v_i)(\psi(x,\dbar) \rightarrow (N(x) \& \bigwedge_{i< n}V(v_i))) =\true$, then $v_B((\forall x)(N(x) \rightarrow \psi(x, \dbar)))=\true$ (by the G-$\omega$-rule).

\end{proof}

The argument requires that $V$ realizes only countably many $\tau^\phi$-types over the empty set; this follows from the Chang's theorem~\ref{mathback}.2 and Definition~\ref{Morcode} since each realized $\tau^\phi$ type is principal.

 \begin{theorem}\label{omegaruleatomic2} Fix a complete $L_{\omega_1,\omega}$-sentence $\phi$. 
With Definition~\ref{ourinfrule2}, if a model of $T^{\hat \phi}$ satisfies the $G-\omega$-rule 
then its  $\tau^\phi$-reduct   is an atomic model of $T^{ \phi}$. 
I.e., each  $p\in S(\emptyset)$ realized in $V$  is principal as a $\tau^\phi$-type.  Thus, by Fact~\ref{mathback}.2, the $\tau$-reducts of models of $T^{\hat \phi} + G-\omega-rule$
are exactly the models of $\phi$. That is, $\phi$ and $T^{\hat \phi} + G-\omega$-rule are structurally equivalent.
\end{theorem}

\begin{proof}
First we note that by the omitting types theorem  the countable atomic model $A$ of $T^\phi$ has an admissible valuation.
Now, for an arbitrary $B \models T^{\hat \phi}$, working in $\tau_B$,
let $v_B$ be an admissible valuation with range $B$ assigning
$c_{n,i}$ for some $n, i<\omega$ to each member of $N(A)$.
For sake of contradiction, let $B$  be a model of  $T^{\phi}$ extending $A$  that is not atomic;  thus for some $n$ there is 
an $n$-tuple $\dbar\in V(B)$
realizing a non-principal $\tau^\phi$-type $p$
over  $\emptyset$. Applying L\"{o}wenheim Skolem, choose a countable elementary submodel $B'$ of $B$ that contains $\dbar$.
By the last line of Definition~\ref{Morcode} and since the elements of $N(A)$ code 
all and only  the principal types of $T^{\phi}$, it is coded by an element $e$ of   $N(B') - N(A)$, i.e. $B'\models R^n(e,\dbar)$.
Since $p$ is non-principal, $B\models \neg R^n(c_{n,i},\dbar) \wedge N(c_{n,i})$ for each standard $c_{n,i}$. Applying the $G-\omega$-rule
with  the formula $\psi(x,\dbar)$ as $\neg R^n(v_0,\dbar)$, $B'\models (\forall v_0) (N(v_0) \rightarrow \neg R^n(v_0,\dbar))$. 
Hence, such a $\dbar$ cannot exist and so $\dbar$ realizes a principal type in $V(B')$ so in  $V(B)$. From the contradiction we conclude $V(B)$ is atomic.

Chang's theorem  shows $\bK_\phi \subseteq \bK_{{T^\phi}}^{\tau}$
(Notation~\ref{classnot}); the converse is immediate  in $G-\omega$-logic since every  model
of the $G-\omega$-rule is atomic.   So $\phi$ and $T^{\hat \phi} + G-\omega$-rule are structurally equivalent.
 \end{proof}

Note that this shows that if $T$ is a first order theory categorical in $\aleph_0$, then the associated $T^\phi$ will satisfy
the G-$\omega$-rule; there are no non-principal types to omit. This applies even to the much studied totally categorical theories
(there exists a unique up to isomorphism model in each cardinality). But none of these theories are categorical in the sense of his paper. We describe the well-known characterization of those structures
that are categorical (in $L_{\omega_1,\omega}$).

\begin{definition}\label{gendef} A structure $M$ is said to be generative if it is isomorphic to a proper substructure (equivalently extension)
of itself.
\end{definition}

Immediately from the structural equivalence (Definition ~\ref{sedef2}, Theorem ~\ref{omegaruleatomic2}):
\begin{corollary}\label{coupdegras} $\phi$ is categorical if and only if $T^{\hat \phi}$ is categorical in $G-\omega$-logic if and only if the countable model of $\phi$ is non-generative.
\end{corollary}
\begin{proof} Only the second equivalence needs argument. By Theorem~\ref{omegaruleatomic2}  both $M$ and $M_1$ are atomic (and without loss countable).
So $M_1 \approx M$.  But then following \cite{Vaughtct} we can build an uncountable chain $\langle M_\alpha: \alpha < \omega_1$ of pairwise isomorphic elementary extensions $\langle M_\alpha:\alpha < \omega_1\rangle$. So $M_{\omega_1}$ is uncountable and violates categoricity.
\end{proof}

Any complete sentence of $L_{\omega_1,\omega}$ is $\aleph_0$-categorical. But it will be categorical only if the countable model is not generative.
In particular, the standard model of arithmetic  is atomic since for any finite sequence  $\abar\in N$,
$\bigwedge_{i<n} v_i =a_i \rightarrow \qtp_{\emptyset}(a_0, \ldots a_{n-1})$; the model is in the definable closure of
the empty set.

The structure $(\Z,+,0)$, the free abelian group on one generator,  is certainly canonical and so a target for a categoricity theorem. Since there is only one constant term, it cannot be proved categorical by the methods
of Section~\ref{subquant}. But is easily axiomatized in $L_{\omega_1,\omega}$ as a torsion free Abelian group generated
by a single element $x$ satisfying $\bigwedge_{w_1,w_2} (w_1(x) = w_2(x)\rightarrow (\forall y) (w_1(y) = w_2(y))$ where
$w_1,w_2$ range over all words  with a single argument.

 Example~\ref{catexs}.2 shows that unlike Peano Arithmetic, there are simple examples of $L_{\omega_1,\omega}$ categorical structures that are not in the definable closure of a finite set.

\section{Conclusion.}

We studied in this paper two main  notions of categoricity: the  categoricity of a logical calculus and the categoricity of a first-order theory formalized in it. Each logical calculus was considered as an inferential logical system with natural deduction rules. By results of Carnap, the standard axiomatic and natural deduction calculi for classical propositional and first-order logic are known to be non-categorical, i.e. they allow non-standard admissible valuations that provide some of the logical connectives and quantifiers with non-standard truth-conditions. 

Our solution for obtaining a categorical formalization of the first-order quantifiers was to introduce a new form of a well-known infinitary $\omega$-rule of inference,  the $I-\omega$-rule. With the $I-\omega$-rule we obtained a categorical formalization of the first-order universal quantifier and, at the same time, we proved that all first-order theories that have an algebraically prime model in which all elements are named by constants are categorical when formalized in the inferential $I-\omega$-logic. Similar, we proved in the two-sorted generalized $\omega$-logic that
 each  complete $L_{\omega_1,\omega}$ sentence defines the same class of structures
as a corresponding first-order theory with the $G-\omega$-rule.
Moreover, we proved (Corollary~\ref{coupdegras}) that a sentence  $\phi$ of $L_{\omega_1,\omega}$ (and so the associated  (Theorem~\ref{omegaruleatomic2}) theory in $G-\omega$-logic) is categorical just if 
the countable model of $\phi$ has no proper isomorphic extension modeling 
$\phi$. As the extended logics satisfy the downward L\"{o}wenheim Skolem theorem, all the categorical structures are countable.

In the  sequel, \cite{BaldwinBrincusII}, we address the doxological challenge  and circularity objections to `first order'  axiomatizations of arithmetic posed by \cite{ButtonWalshbook}.

\bibliography{cbjb.bib}

@book{Baldwincatmon,
author ={John T. Baldwin},
title = "Categoricity" ,
year ={2009},
publisher ="American Mathematical Society, Providence, USA",
series = "University Lecture Notes",
note={The monograph is free at: http:
//homepages.math.uic.edu/˜jbaldwin/pub/AEClec.pdf},
number = "51"
}

@article{BaldwinBrincusII,
  author  = {Baldwin, John T. and Br{\^i}ncu{\c{s}}, Constantin C.},
  title   = {Carnapian Frameworks and Categoricity of Arithmetic via Inferential $\omega$-logics},
  year    = {2026},
  note    = {Manuscript} 
}

@book{modthelog,
TITLE = " Model-Theoretic Logics",
publisher = " Springer-Verlag",
year = {1985},
editor =  "Barwise,J. and Feferman,S."
}

@incollection{Bellsep,
    author = "Bell, John L.",
    title = "Infinitary Logic",
    booktitle = "The Stanford Encyclopedia of Philosophy (Fall 2023 Edition)",
    editor  ="Zalta, Edward N.  and Nodelman, Uri ",
    publisher ="https://plato.stanford.edu/archives/fall2023/entries/logic-infinitary/",
    year = "2023"}

@article{BrincusS,
title ="Are the open-ended rules for negation categorical?",
author ="Br{\^i}ncu{\c{s}}, Constantin C.",
journal ="Synthese",
volume = {198},
pages ={7249–7256},
year ={2021},}

@article{Brincus,
title ="Categorical Quantification",
author ="Br{\^i}ncu{\c{s}}, Constantin C.",
journal ="The Bulletin of Symbolic Logic",
volume = {30},
pages ={227-252},
year ={2024},}

@incollection{Brincus2,
author = "Br{\^i}ncu{\c{s}}, Constantin C.",
title = "Inferential Quantification and the $\omega$-Rule",
booktitle = "Perspectives on Deduction: Contemporary Studies in the Philosophy, History and Formal Theories of Deduction",
year ="2024",
editor ={Piccolomini d'Aragona, A.},
publisher = {Synthese Library, vol 481. Springer, Cham},
pages ="345-372",}

@book{ButtonWalshbook,
title ="Philosophy and Model Theory",
year = {2018},
author ="Button, Tim and Walsh, Sean ",
publisher ={Oxford University Press},
pages = {xvi+517}
}

@book{CarnapLog,
title = "Logical Syntax of Language",
author = {Carnap, Rudolf},
publisher = "Trench, Trubner and Co Ltdr",
address ={London},
year =  "1934/37"}

@book{CarnapFor,
title ="Formalization of Logic",
author ={Carnap, Rudolf},
publisher ="Harvard University Press",
address ={Cambridge},
year = "1943"}

@incollection{Chang1,
author ={C. C. Chang},
title =" Some Remarks on the model theory of Infinitary Languages",
booktitle ="The syntax and semantics of infinitary languages",
editor = {J.~Barwise },
pages ={36-64},
publisher = {Springer-Verlag},
note ={LNM 72},
year ={1968}}

@book{ChangKeisler,
author ="Chang,C.C. and Keisler, H.J.",
year = "1990",
title = "Model theory",
publisher = "North-Holland",
note ="3rd impression 1992",
pages ="XII+550"}

@book{Dedekind,
title = "Essays on the Theory of Numbers",
year = "1888/1963"                               ,
publisher = "Dover"                 ,
address ={Mineola, New York},
author = "Richard Dedekind" ,
note = "first published by Open Court publications 1901: first German edition 1888, translated
from 2nd German edition 1893 which contains the study of arithmetic",
pages = "115"}

@incollection{Ebbing,
author ={Ebbinghaus, H.D.},
title ="
Extended Logics: The General Framework",
booktitle ="Model-Theoretic Logics", 
editor = {J.~Barwise and S.~Feferman},
pages ={25-76},
publisher = {Springer-Verlag},
year ={1985}}

@incollection{Eastaughsep,
    author = "Eastaugh, Benedict",
    title = "Reverse Mathematics",
    booktitle = "The Stanford Encyclopedia of Philosophy (Summer 2024 Edition)",
    editor  ="Zalta, Edward N.  and Nodelman, Uri ",
    publisher ={\url{https://plato.stanford.edu/archives/sum2024/entries/reverse-mathematics/}},
    year = "2024"
}

@book{Forbes,
title ="Modern Logic. A Text in Elementary Symbolic Logic",
author ={Forbes, Graeme},
publisher ="Oxford University Press",
address ={New York},
year = "1994"}

@article{Franzen,
AUTHOR = "Torkel Franz{\'e}n",
YEAR = "2004" ,
TITLE = "Transfinite Progressions: A Second look at completeness",
journal = {Bulletin of Symbolic Logic },
VOLUME = {10} ,
PAGES = {367-389}
}

@incollection{Gentzen,
author ="Gerhard Gentzen",
title="Investigations into Logical Deduction",
pages = "68-131",
booktitle = "The Collected Papers of Gerhard Gentzen",
year = "1969"                               ,
publisher = "North-Holland, Amsterdam",
note ="German original published in 1934",
editor = "M. E. Szabo "}

@book{Garson,
title ={What Logics Mean: From Proof-Theory to Model-Theoretic Semantics},
author = {Garson, James},
publisher ="Cambridge University Press",
address = {Cambridge},
year = "2013"}

@article{Henkinomt,
author = " Leon Henkin",
title = "A generalization of the notion of $\omega$-consistent",
JOURNAL = {Journal of Symbolic Logic},
volume = {19},
pages = "183-196",
year = "1954"}

@unpublished{Ketland,
title="Completeness of {PA} with omega rule",
author="Ketland, Jeffrey",
year = "2011",
note = {\url{https://m-phi.blogspot.com/2011/03/completeness-of-pa-with-omega-rule.html}}
}

@book{Halbach,
    author ="Volker Halbach" ,
    title = "The Logic Manual",
    publisher = "Oxford University Press" ,
    year = "2010"
}

@incollection{Isaacson,
    author = "Daniel Isaacson",
    title = "Some considerations on arithmetical truth and the $\omega$-rule",
    editor = "M. Detlefsen",
    booktitle = "Proof, logic and formalization",
    publisher = "Routledge",
    year = "1992",
    pages ="94-138", }

@book{Keislerbook,
author =" Keisler,H.J",
year = "1971",
title = "Model theory for Infinitary Logic",
publisher = "North-Holland",
pages ="X+ 208"
}

@book{Kunen,
title = "Set {T}heory, {A}n {I}ntroduction to {I}ndependence {P}roofs",
year = " 1980"                               ,
publisher = "North Holland",
author = "K. Kunen"}

@article{Lavine,
    author = "Shaughan Lavine" ,
    title = " Quantification and Ontology",
    journal = " Synthese",
    volume = " 124",
    pages =" 1-43",
    year = " 2000",}

@book{Markerbook,
author ="Marker,D.", 
title ="Model Theory: An {I}ntroduction", 
publisher ="Springer-Verlag", 
year ="2002" }

@book{Markerinflog,
title ={Lectures on Infinitary Model Theory},
author={Marker,D.},
series ={Lectures Notes in Logic},
publisher = {Association of Symbolic Logic, Cambridge University Press},
address = {Cambridge},
year = {2016},
pages ={viii +183}
}

@incollection{Morley65a,
author = "Morley, M.",
title = "Omitting Classes of Elements",
booktitle = "The Theory of Models",
year =" 1965",
editor = "Addison and  Henkin and Tarski",
publisher = "North-Holland",
address ="Amsterdam",
pages ="265-273"}

@incollection{MorleyLeeds,
author = "Morley, M.",
title = "Partitions and Models",
booktitle = "Proceedings of the Summer School in Logic Leeds, 1967",
year =" 1968",
editor ={M.H. L\"{o}b},
note= {paper written by Vivienne Morley},
publisher = {Springer-Verlag},
addresss ={Berlin-Heidelberg-New York},
pages ="109-158"}

@article{Orey,
author = " S. Orey"   ,
year = "1956"                          ,
title = "On $\omega$-consistency and related properties",
journal = "Journal of Symbolic Logic",
volume = {21 },
pages =" 246-252    "
}

@incollection{PelletierHazen,
    author = " Pelletier, Francis Jeffry and Hazen, Allen P.",
    title = "A History of Natural Deduction",
    booktitle = "Handbook of the History of Logic",
    volume = "11; Logic: A Hystory of its Central Concepts",
    editors = "Dov M. Gabbay, Francis Jeffry Pelletier, John Woods",
    publisher = "Elsevier, North-Holland",
    year = "2012",
    pages = "341-414"
}

@article{Rosser,
author = "Rosser, Barkley"   ,
year = "1937"                          ,
title = "Gödel Theorems for Non-Constructive Logics",
journal = "Journal of Symbolic Logic",
volume = {2:3},
pages ="129-137"
}

@book{Sacks,
    author = {Sacks, G.E.},
    title = {Saturated Model Theory},
    publisher = {Reading, Massachusetts: W A Benjamin},
    year = {1972}
}

@book{Shelahbook               ,
author = " Shelah,S."           ,
year = " 1978"                   ,
title = " Classification
{T}heory and the {N}umber of {N}onisomorphic {M}odels" ,
PUBLISHER = " North-Holland"  ,
pages = " XVI+544"              ,
reviews = " MR 81a:03030"
}

@incollection{Shapiro,
booktitle = "Handbook of Philosophical Logic",
author = {Shapiro, Stewart },
title = {Systems between First-Order and Second-Order Logics},
volume = {1},
year = " 2001"                               ,
editor = "D.M. Gabbay, F. Guenthner (eds) ",
publisher = "Springer",
address ={Dordrecht},
pages = "131-187",
}

@book{Shoenfield,
 title ="Mathematical {L}ogic",
  author ={Joseph Shoenfield},
  year ={1967},
  publisher ={Addison-Wesley},
   pages ={344}}

@article{Shoenfieldomega,
author ="Shoenfield,  Joseph R.",
title ="On a restricted $\omega$-rule",
journal= "Bull. Acad. Polon. Sci. Ser. Sci. Math. Astr. Phys.",
year = {1959},
volume ={ 7},
pages ={405-407}}

@book{ShoSmiley,
    author = {D.J. Shoesmith and T.J. Smiley},
    title = {Multiple-Conclusion Logic} ,
    publisher = "{Cambridge Unversity Press}",
    year = {1978}
}

@book{Simpson             ,
author = " Simpson ,S.G."         ,
year= "2009"                   ,
title = " Subsystems of Second Order Arithmetic",
PUBLISHER = "Cambridge"  ,
pages = "              xvi + 444"
}

@article{Skolem34,
title ="{\"U}ber  die Nicht-charakterisierbarkeit der Zahlenreihe mittels endlich oder ab{\"a}hlbar unendlich vieler Aussagen
 mit ausschliesslich Zahlenvariablen",
author ="Skolem,  Thoralf",
journal ="Fundamenta Mathematicae",
volume ={23},
year ={1934},
pages={150-161}
}

@book{Smith,
    author = "Peter Smith",
    title = "An Introduction to Formal Logic, 2nd edition",
    publisher = "Cambridge University Pres",
    note = "Reprinted with corrections, Logic Matters, August 2020, https://www.logicmatters.net/ifl/",
    year = "2020"}

@article{Speitel,
    author = "Speitel, Sebastian G.W.",
    title = "Carnap’s (Categoricity) Problem",
    journal = "The Bulletin of Symbolic Logic",
    note = "10.1017/bsl.2025.10083",
    year = "2025",    
}

@book{TMR     ,
author = "Tarski, Alfred and   Mostowski, Andrezej and  Robinson, Raphael",
title = "Undecidable Theories",
publisher ={North Holland},
note ="First edition: Oxford 1953",
year = "1968"
}

@article{TarskiVaught      ,
author = "A. Tarski and R.L. Vaught",
title = "Arithmetical extensions of relational systems  ",
volume = "13"   ,
pages = " 81-102",
journal = " Compositio Mathematica",
year = "1956"
}

@article{Vaughtct,
author = " Vaught, R.L.",
title = " Models of complete theories",
journal= "Bull. Amer. Math. Soc.",
publisher = "American Mathematical Society",
editior = "",
NOTE = { \url{http://homepages.math.uic.edu/~jbaldwin/pub/vaught59.pdf}},
pages = "299-313",
year = "1963"
}
\bibliographystyle{alpha}
%
%
%
\end{document}